%% file: root.tex
\documentclass[twocolumn,10pt]{asme2ej}

\usepackage{graphicx} 

\usepackage{cite}
\usepackage{amsmath,amssymb,amsfonts}
\usepackage{bm}
\usepackage{algorithmic}
\usepackage{graphicx}
\usepackage{textcomp}
\usepackage{xcolor}
\newtheorem{assumption}{Assumption}   
\newtheorem{lemma}{\bf{Lemma}}
\newtheorem{definition}{\bf{Definition}}

\newtheorem{theorem}{\bf{Theorem}}

\title{Hierarchical Robust Adaptive Control for Wind Turbines with Actuator Fault}

\author{Sina Ameli and Olugbenga Moses Anubi$^*$
    \affiliation{
	Department of Electrical Engineering\\
	Florida State University\\
	Tallahassee, Florida 32310\\
    Email: sa19bk@my.fsu.edu and oanubi@fsu.edu
    }	
}

\begin{document}

\maketitle    

\begin{abstract}
{\it This paper solves the problem of regulating the rotor speed tracking error for wind turbines in the full-load region by an effective robust-adaptive control strategy. The developed controller compensates for the uncertainty in the control input effectiveness caused by a pitch actuator fault, unmeasurable wind disturbance, and nonlinearity in the model. Wind turbines have multi-layer structures such that the high-level structure is nonlinearly coupled through an aggregation of the low-level control authorities. Hence, the control design is divided into two stages. First, an $\mathcal{L}_2$ controller is designed to attenuate the influence of wind disturbance fluctuations on the rotor speed. Then, in the low-level layer, a controller is designed using a proposed adaptation mechanism to compensate for actuator faults. The theoretical results show that the closed-loop equilibrium point of the regulated rotor speed tracking error dynamics in the high level is finite-gain $\mathcal{L}_2$ stable, and the closed-loop error dynamics in the low level is globally asymptotically stable. Simulation results show that the developed controller significantly reduces the root mean square of the rotor speed error compared to some well-known works, despite the largely fluctuating wind disturbance, and the time-varying uncertainty in the control input effectiveness.
}
\end{abstract}

\section{Introduction}\label{sec:introduction}
\input{Introduction}

\section{Notation and Preliminary}\label{sec:prelim}
\input{Preliminaries}

\section{Problem Formulation}\label{sec:problem}
\input{Problem_Formulation}
\section{Control Development}\label{sec:control}
\input{Control_development}

\section{Numerical Results}\label{sec:simulation}
\input{Simulation}

\section{Conclusion}\label{sec:conclusion}
\input{Conclusion}
\section{Future Work}\label{sec:future work}
\input{Future_work}

%
% ...

\bibliographystyle{asmems4}
\bibliography{asme2e}

\end{document}

%% file: Introduction.tex
Wind energy is increasingly becoming a significant source of clean energy. However, its market penetration has been affected adversely by high costs related to manufacturing, operation, and maintenance, especially for megawatt size {\color{blue} wind turbines} {\color{blue}(}WTs{\color{blue})} producing a great deal of electricity. To expand the lifetime of such large WTs, deploying active load alleviation control is critical. One efficient approach to decrease structural loading, and thereby extend the lifetime of the WTs while maintaining the rated electrical power is regulating the rotor speed tracking error via controlling pitch angles~\cite{bianchi2006wind,palejiya2015performance,imran2014dac,kalbat2013linear}.
 In~\cite{palejiya2015performance} a PI controller is designed for a WT to alleviate the mechanical stress; however, it highly relies on gain selection based on the wind speed. In~\cite{imran2014dac} a linear quadratic regulator is designed to track the generator speed. However, the closed loop system performance is downgraded in wider operating range. In~\cite{kalbat2013linear}, the authors proposed a linear quadratic Gaussian control. However, the controller suffers from lack of robustness to model uncertainties.
 
Since WTs are aero-electro-mechanical systems, their dynamics are surpassingly complicated. There is no explicit relation between the rotor speed and the pitch angle for WTs in the region that the wind speed is above the rated value (full-load region). They have nonlinear dynamics exposed to stochastic wind signal changing the WTs' operating points. Thus, these issues necessitate the design of the nonlinear controllers for WTs. In~\cite{jafarnejadsani2013l1} a gain-scheduling optimal controller is proposed utilizing genetic algorithm for pitch control; however, it can only deal with a limited types of wind signals. In~\cite{sarkar2020nonlinear} a nonlinear model predictive control (MPC) is designed for a WT to reduce pitch actuation while reducing the mechanical loads and regulating the rotor speed tracking error.
However, the controller is not robust against model uncertainty.

Moreover, it has been proven that robust and adaptive control strategies tackle with the wind disturbance fluctuations, model uncertainty, and time-varying actuator faults in WTs~\cite{qi2014h,benlahrache2017fault,jiao2020ewse,ma2014moving,simani2012adaptive,simani2013data,wang2020adaptive,zuo2015l2}. In~\cite{qi2014h}, an $\mathcal{H}_{\infty}$-state-feedback controller is designed for a WT to mitigate actuator faults. However, the controller's performance is limited to specific operating points. In~\cite{jiao2020ewse} a disturbance observer and effective wind speed estimator are designed for a WT with model uncertainty to maintain the rated electrical power by controlling the pitch angle. However, the controller design relies on the estimation of the wind speed. In~\cite{benlahrache2017fault} a robust MPC is designed by solving an LMI for a WT while the pitch actuator is subjected to a fault. However, the controller's performance is dependent on the linearized model of the WT.  In~\cite{ma2014moving} a moving horizon $\mathcal{H}_{\infty}$ controller is designed for a WT to control the pitch angle. It shows that the closed-loop system achieves the $\mathcal{L}_2$ disturbance attenuation from the wind signal to the states in the framework of dissipation theory. However, the pitch actuators are assumed to be healthy during their operations. In~\cite{simani2012adaptive}, and~\cite{simani2013data} an online system identifier using adaptive directional forgetting scheme is designed to extract the WT model and approximate it to a second or third order system for pitch and torque channels. Then an adaptive PI controller is designed using modified Ziegler-Nichols rules. Although the controller can estimate the time-varying parameters, and regulates the rotor speed error, it is not robust against the wind signal with high turbulence intensity. In~\cite{wang2020adaptive} a robust adaptive controller is designed using clustering-type fuzzy neural network for a WT subjected to pitch actuator faults. However, the fuzzy controller itself has several parameters (defuzzification, inference, input and output membership functions, rulebases, etc) to be determined. In~\cite{zuo2015l2} an adaptive $\mathcal{L}_2$-gain controller is designed for a class of singular systems with Lipschitz nonlinearity, actuator saturation, and actuator faults. Although the results show that the closed-loop system achieves smaller upper bound on the $\mathcal{L}_2$-gain, the actuator faults are assumed to be constant.
The major contributions of this paper are (i) a multi-layered control structure consistent with the nature of the system, (ii) a detailed finite-gain analysis and design to compensate for unmeasurable wind disturbance, nonlinearity and model uncertainty in the high-level layer, and (iii) low-level adaptive control design to compensate for time-varying incipient pitch actuator faults. 
The remaining of the paper is organized as follows: Section~\ref{sec:prelim}, introduces notations and preliminary. Section~\ref{sec:problem}, presents the system model. Section~\ref{sec:control}, illustrates the control development for all layers. Section~\ref{sec:simulation}, shows the numerical simulation results. Section~\ref{sec:conclusion}, provides conclusion remarks, and finally the future work is discussed in section~\ref{sec:future work}.

%% file: Preliminaries.tex
The following notions and conventions are utilized in the paper:
$\mathbb{R}$, and $\mathbb{R}^n$ denote the space of real numbers, and real vectors of length $n$, respectively. $\mathbb{R}_+$ includes zero and positive real numbers.
Normal-face lower-case letters ($x\in\mathbb{R}$) are used to represent real scalars, bold-face lower-case letter ($\mathbf{x}\in\mathbb{R}^n$) represents vectors. The euclidean balls $\mathbb{B}_r(0)$ is defined for some $r>0$ as $\mathbb{B}_r(0)\triangleq\left\{\textbf{x}:\left\|\textbf{x}\right\|\le r\right\}$. Moreover, $\mathcal{L}_2$ is the space of all piecewise continuous, square-integrable functions $\mathbf f:\mathbb{R}_+\mapsto \mathbb{R}^n$, and the $\mathcal{L}_2$ norm of $\mathbf{f}\in\mathcal{L}_2$ is defined by 
\begin{align*}
    \left\|\textbf{f}\right\|_2\triangleq\left(\int_0^\infty{\|\textbf{f}(\tau)\|^2d\tau}\right)^{\frac{1}{2}}<\infty,
\end{align*}
and the extended space $\mathcal{L}_{2e}$ is the space of all measurable functions $\mathbf{f}$ such that $\mathbf{f}_T\in\mathcal{L}_2$ where
\begin{align*}
\mathbf{f}_T(t)\triangleq\left\{\begin{array}{rl}\mathbf{f}(t),&0\le t<T\\\textbf{0},&t\ge T\end{array}\right.
\end{align*}
, for all $T\in[0,\infty)$.
\begin{definition}($\gamma$-dissipativity)\cite{van2000l2}\label{Def:disp}
Consider the nonlinear system
\begin{align} \label{eqn:General}
   \mathcal{G}:\hspace{5mm} \begin{array}{rl}
    \dot{\mathbf{z}}&=g(\mathbf{z},\boldsymbol{\nu})\\
          \boldsymbol{\omega}&=h(\mathbf{z})
    \end{array}
\end{align}
where $\boldsymbol{\nu}(t)\in\mathcal{L}_{2e}^q$, $\mathbf{z}(t)\in \mathcal{L}_{2e}^n$, $\boldsymbol{\omega}(t)\in\mathcal{L}_{2e}^s$ are the input, the state variable, and the output, respectively. The nonlinear system in Eqn.~\eqref{eqn:General} is dissipative with respect to the supply rate $\ell:\mathcal{L}_{2e}^q\times\mathcal{L}_{2e}^s \mapsto\mathcal{L}_{2e}$, if an energy function $V(\mathbf{z})\geq 0$ exists such that, for all $t_1\geq t_0$, 
\begin{align}\label{eqn:dissipativity ineq}
    V(\mathbf{z}(t_1))\leq V(\mathbf{z}(t_0))+\int_{t_0}^{t_1} \ell(\boldsymbol{\nu},\mathbf{\omega}) d\tau\hspace{2mm}\text{for all}\hspace{2mm} \boldsymbol{\nu}\in\mathcal{L}_{2e}^q.
\end{align}
In addition, for a $\gamma>0$, if the supply rate is chosen as $\ell(\boldsymbol{\nu},\boldsymbol{\omega})=\gamma^2\left\|\boldsymbol{\nu}\right\|_2^2-\left\|\boldsymbol{\omega}\right\|_2^2$,
then \eqref{eqn:dissipativity ineq} indicates a finite-gain $\mathcal{L}_2$ stability. As a result, the dynamic is $\gamma$-dissipative and inequality \eqref{eqn:dissipativity ineq} yields $\dot{V}\leq\gamma^2\left\|\boldsymbol{\nu}\right\|_2^2-\left\|\boldsymbol{\omega}\right\|_2^2$.
\end{definition}
\begin{lemma} \label{lem:Lemma1}
The following characteristics are satisfied\cite{anubi2013variable}\\
\begin{enumerate}
    \item $\textsf{sat}(\chi\mathbf{y},y_1,y_2)=\chi\textsf{sat}(\mathbf{y},\frac{y_1}{\chi},\frac{y_2}{\chi}), \text{ for all } \chi>0$
    \item For $r>0$, there exists $\mu>0$ such that $\mathbf{y}^\top\textsf{sat}(\mathbf{y},y_1,y_2)\geq \mu \left\|\mathbf{y}\right\|_2^2, \text{ for all } \mathbf{y}\in \mathbb{B}_r(0)$
\end{enumerate}
\end{lemma}

%% file: Problem_Formulation.tex
In general WTs are operational in two areas, namely partial-load and full-load regions shown Figure~\ref{fig:WT regions}. In the former the optimum electrical power is generated via a torque control. However, in the latter, since the wind speed is above the rated value, the rated electrical power can be generated but it should not pass its rated value and meanwhile the mechanical stress should be reduced via the pitch control.
\subsection{Drive-train Dynamics}
\begin{figure}[t]
    \centering
    \includegraphics[width=8cm, height=4.5cm]{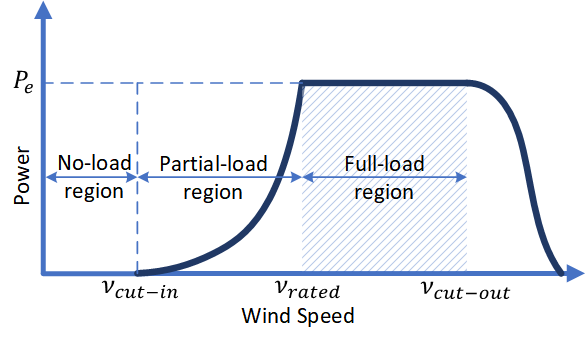}
    \caption{Electrical power generated by a typical WT at different regions}
    \label{fig:WT regions}
\end{figure}

The rotor dynamics in WTs is provided by \cite{jonkman2009definition,bianchi2006wind,wasynczuk1981dynamic}
 \begin{equation}\label{eqn:rotor_dyn}
     \dot{\omega}_r=f(\nu,\omega_r,\theta_j),\hspace{5mm} j=1,2,3
 \end{equation}
 ,where $\omega_r\in \mathcal{L}_{2e}$, is the rotor speed, $\theta_j\in\mathbb{R}_+$ is the $j$th pitch angle, $\nu\in\mathcal{L}_{2e}$ is the wind disturbance, and 
 \begin{align*}
     f(\nu,\omega_r,\theta_j)=g_1(\nu,\omega_r)-g_2(\nu,\omega_r)\left\|\boldsymbol{\theta}\right\|^2,
 \end{align*}
 where $\boldsymbol{\theta}\in\mathbb{R}^3$, and
 \begin{align}\label{eqn:g1}
     g_1(\nu,\omega_r)=&\frac{\kappa \nu^3}{J\omega_r}\left(\frac{\nu}{\omega_r}-p_1\right)\textsf{e}^{(-p_2\frac{\nu}{\omega_r})}-\frac{P_0}{J\omega_r},
     \end{align}
\begin{align}\label{eqn:g2}
       g_2(\nu,\omega_r)=&\frac{\kappa \nu^3}{3J\omega_r}p_3\textsf{e}^{(-p_2\frac{\nu}{\omega_r})}, 
\end{align}     
where $\kappa$ is a positive constant, and the positive constants $p_1$, $p_2$, and $p_3$ are obtained experimentally. $P_0$, and $J$ are the rated mechanical power, and the total drive-train inertia, respectively.
\subsection{Pitch Systems}
The pitch actuators include a hydraulic system whose dynamics are expressed as follows~\cite{odgaard2013wind}
  \begin{align} \label{equ:2nd order pitch Act}
    \Ddot{\theta}_j=-2\zeta\omega_n\dot{\theta}_j-\omega^2_n\theta_j+\omega_n^2\theta_{rj}, \hspace{1cm} j=1,2,3,
\end{align}
 where $\zeta$, is the damping ratio, and $\omega_{n}$ is the natural frequency. $\theta_j$, and $\theta_{rj}$ are the pitch angle output and the pitch reference (control input) for the $j$-th pitch actuator, respectively.
 \begin{assumption}
The pitch angle $\theta_j$ is mechanically constrained such that it is bounded as $0\leq\theta_j\leq\overline{\theta}$, where $\overline{\theta}$ is the maximum allowable pitch angel.
\end{assumption}
 \subsubsection{Fault Model}
One type of fault that could happen in pitch actuators is the high-air content fault, which is an incipient fault, making the oil is mixed with air. The occurrence of this fault changes the characteristics of the second-order dynamics in Eqn.~\eqref{equ:2nd order pitch Act}. Thus, the uncertainty caused by this fault is parametric modeled as follows
\begin{align}\label{equ:fault-model}
    \omega_{n}^2&=\delta\omega_{n_0}^2,\\\nonumber
    \zeta\omega_{n}&=\rho\zeta_0\omega_{n0},
\end{align} 
where $\omega_{n_0}$, and $\zeta_0$ are the operating points (faultless situation). Moreover, $\delta$ and $\rho$ are uncertain parameters because of the fault but they are bounded as $0<\delta\leq 1$, $0<\rho\leq 1$~\cite{odgaard2013wind}.

%% file: Control_development.tex
In this section, the aim is to track the rotor speed operating point $\omega_{r0}$ robustly.
This is achieved via $\gamma$-dissipativity. Define the rotor speed tracking error
\begin{align}\label{equ:rotor error}
    \widetilde{\omega}_r=\omega_r-\omega_{r0}.
\end{align} Denote the following filtered error
\begin{align}\label{equ:s-surface1}
   \sigma=&\widetilde{\omega}_r+\psi\underbrace{\int_0^t\widetilde{\omega}_r d\tau}_{\text{$\widetilde{\omega}_{rI}$}},
\end{align}
where $\psi>0$.
Taking first time derivative of Eqn.~\eqref{equ:s-surface1}, and then inserting the time derivative of Eqn.~\eqref{equ:rotor error} into it yields
\begin{align}\nonumber
 \dot{\sigma}=&\dot{\widetilde{\omega}}_r+\psi\dot{\widetilde{\omega}}_{rI}\\
    =&f(\nu,\omega_r,\theta_j)+\psi\widetilde{\omega}_r\label{equ:High_Op}.
\end{align}
Consider the equilibrium points $\omega_{r0}$, $\nu_0$, and $\theta_0$ in Eqn.~\eqref{eqn:rotor_dyn}, then  $f(\nu_0,\omega_{r0},\theta_0)=0$. Subtracting this equation from Eqn.~\eqref{equ:High_Op} yields
\begin{align*}
    \dot{\sigma}=f(\nu,\omega_r,\theta_j)-f(\nu_0,\omega_{r0},\theta_0)+\psi\widetilde{\omega}_r.
\end{align*}
Then, invoking the mean value theorem~\cite{rudin1964principles} results in
\begin{align}\label{equ:Ol}
\dot{\sigma}=\rho_{\nu}\widetilde{\nu}+\left(\rho_{\omega}+\psi\right)\widetilde{\omega}_r+\boldsymbol{\rho}_{\theta}^\top\boldsymbol{\widetilde{\theta}}
\end{align}
where
\begin{align*}
  \rho_{\nu}=&\frac{\partial f(\eta_{\nu},\eta_{\omega},\boldsymbol{\eta}_{\theta})}{\partial \nu},\\
  \rho_{\omega}=&\frac{\partial f(\eta_{\nu},\eta_{\omega},\boldsymbol{\eta}_{\theta})}{\partial \omega},\\
  \boldsymbol{\rho}_{\theta}=&\nabla_{\theta}f(\eta_{\nu},\eta_{\omega},\boldsymbol{\eta}_{\theta}),
\end{align*}
with $\eta_{\nu}=\lambda\nu_0+(1-\lambda)\nu$, $\eta_{\omega}=\lambda\omega_{r0}+(1-\lambda)\omega_r$, $\boldsymbol{\eta}_{\theta}=\lambda\boldsymbol{\theta}_0+(1-\lambda)\boldsymbol{\theta}$ for some $\lambda\in(0,1)$, and
\begin{align*}
    \widetilde{\boldsymbol{\theta}}&=\boldsymbol{\theta}-\boldsymbol{\theta}_0,\\
    \widetilde{\nu}&=\nu-\nu_0
\end{align*}
, where $\boldsymbol{\theta}_0$, and $\nu_0$ are the pitch angle vector, and the wind speed at the operating point. These operating points are known, and they are usually used in the component design and rating. So, it is reasonable to assume that they are known. Moreover, the simulation results in section~\ref{sec:simulation}, will show that the controller regulates the rotor speed tracking error well at different operating points.

Here, the objective is to design an auxiliary control law $\widetilde{\boldsymbol{\theta}}$ for Eqn.~\eqref{equ:Ol}, such that the rotor speed tracking error $\widetilde{\omega}_r$ is robustly regulated for all $\widetilde{\nu}\in\mathcal{L}_{2e}$. This is then deployed for a desired pitch angle for the low-level dynamics. The subsequent properties hold for the open-loop dynamics in Eqn.~\eqref{equ:Ol}.
\begin{assumption}\label{as:bounded uncertainties}
The high-level dynamics is sufficiently smooth. Hence, the uncertain parameters $\rho_{\nu}$, and $\rho_{\omega}$ in Eqn.~\eqref{equ:Ol} are bounded as follows
\begin{align*}
    \left|\rho_{\nu}\right|\leq\bar{\rho}_{\nu},\\
        \left|\rho_{\omega}\right|\leq\bar{\rho}_{\omega},
\end{align*}
where $\bar{\rho}_{\nu}$, and $\bar{\rho}_{\omega}$ are positive constants. Moreover, there exists $\boldsymbol{\rho}_0\in\mathbb{R}^3$, and a positive constant $\varphi$ such that
\begin{align}\label{equ:conic}
    \boldsymbol{\rho}_{\theta}^\top\boldsymbol{\rho}_0\geq \varphi.
\end{align}
\end{assumption}
Consequently, the design of the auxiliary control law is as follows
\begin{align}\label{equ:High-level control law}
    \widetilde{\boldsymbol{\theta}} &= -\boldsymbol{\rho}_0\textsf{sat}\left(k\sigma,-\theta_0,\overline{\theta}-\theta_0\right),
\end{align}
where $k$ is a positive control gain, and $\boldsymbol{\rho}_0$ satisfies the conic constraint in Eqn.~\eqref{equ:conic}. Then, the corresponding high-level closed-loop error system is obtained as follows
\begin{align}\nonumber
    \dot{\sigma}=&\left(\rho_{\omega}+\psi\right)\widetilde{\omega}_r+\rho_{\nu}\widetilde{\nu}-\boldsymbol{\rho}_{\theta}^\top\boldsymbol{\rho}_0\textsf{sat}\left(k\sigma,-\theta_0,\overline{\theta}-\theta_0\right)\\\label{equ:cle}
    =&\left(\rho_{\omega}+\psi\right)\left(\sigma-\psi\widetilde{\omega}_{rI}\right)+\rho_{\nu}\widetilde{\nu}\\\nonumber
    &-\boldsymbol{\rho}_{\theta}^\top\boldsymbol{\rho}_0\textsf{sat}\left(k\sigma,-\theta_0,\overline{\theta}-\theta_0\right)
    \end{align}
    The following theorem provides the robust performance of the auxiliary control law in the high-level layer.
\begin{theorem}
Given the high-level control law in Eqn.~\eqref{equ:High-level control law}, and $\gamma>0$, if the control gain is chosen such that the following sufficient condition is satisfied
\begin{align}\label{equ:SC}
    k\geq \frac{1}{\mu\varphi}\left(1+\frac{\left(\bar{\rho}_\omega+2\psi\right)^2}{4\psi}+\frac{\bar{\rho}_{\nu}^2}{4\gamma^2}\right),
\end{align}
then the corresponding closed-loop error dynamics in~\eqref{equ:cle} is $\mathcal{L}_2$-gain stable and the $\mathcal{L}_2$-gain from the exogenous signal $\widetilde{\nu}$ to the regulated error $\widetilde{\omega}_r$ is upper bounded by $\gamma$.
\end{theorem}
\begin{proof}
Denote the positive definite energy function
\begin{align}\label{equ:lyp_H}
    V=\frac{1}{2}\sigma^2+\frac{\psi^2}{2}\widetilde{\omega}_{rI}^2.
\end{align}
Since $\left\|\widetilde{\omega}_r\right\|_2\leq \left\|\left(1+\frac{1}{s}\widetilde{\omega}_r\right)\right\|_2=\left\|\sigma\right\|_2$, and considering Definition.~\eqref{Def:disp}, it is sufficient to show that $\dot{V}\leq\gamma^2\widetilde{\nu}^2-\sigma^2$.
Taking the time derivative of $V$, and inserting the high-level closed-loop dynamics in Eqn.~\eqref{equ:cle} into it, leads to
\begin{align*}
      \dot{V}&=\sigma\dot{\sigma}+\psi^2\widetilde{\omega}_{rI}\dot{\widetilde{\omega}}_{rI}\\
      &=\sigma\left(\left(\rho_{\omega}+\psi\right)\sigma-\psi\left(\rho_{\omega}+\psi\right)\widetilde{\omega}_{rI}+\rho_{\nu}\widetilde{\nu}\right)\\
      &-\frac{1}{k}\boldsymbol{\rho}_{\theta}^\top\boldsymbol{\rho}_0(k\sigma)\textsf{sat}\left(k\sigma,-\theta_0,\overline{\theta}-\theta_0\right)+\psi^2\widetilde{\omega}_{rI}\left(\sigma-\psi\widetilde{\omega}_{rI}\right)
\end{align*}
Adding, and subtracting $\left(\gamma^2\widetilde{\nu}^2-\sigma^2\right)$, and finally applying Assumption.~\eqref{as:bounded uncertainties}, and Lemma.~\eqref{lem:Lemma1} yield
\begin{align*}
    \dot{V}&=\left(\rho_{\omega}+\psi+1\right)\sigma^2-\psi\rho_{\omega}\sigma\widetilde{\omega}_{rI}+\rho_{\nu}\sigma\widetilde{\nu}\\
    &-\frac{1}{k}\boldsymbol{\rho}_{\theta}^\top\boldsymbol{\rho}_0(k\sigma)\textsf{sat}\left(k\sigma,-\theta_0,\overline{\theta}-\theta_0\right)\\
    &-\psi^3\widetilde{\omega}_{rI}^2-\gamma^2\widetilde{\nu}^2+\left(\gamma^2\widetilde{\nu}^2-\sigma^2\right)\\
    &=\left(\rho_{\omega}+\psi+1+\frac{\rho_{\omega}^2}{4\psi}+\frac{\rho_{\nu}^2}{4\gamma^2}\right)\sigma^2-\psi\left(\psi\widetilde{\omega}_{rI}+\frac{\rho_{\omega}}{2\psi}\sigma\right)^2\\
    &-\gamma^2\left(\widetilde{\nu}-\frac{\rho_{\nu}}{2\gamma^2}\sigma\right)^2-\frac{1}{k}\boldsymbol{\rho}_{\theta}^\top\boldsymbol{\rho}_0(k\sigma)\textsf{sat}\left(k\sigma,-\theta_0,\overline{\theta}-\theta_0\right)\\
    &+\left(\gamma^2\widetilde{\nu}^2-\sigma^2\right)\\
    &\leq-\left(k\mu\varphi-1-\frac{\left(\bar{\rho}_\omega+2\psi\right)^2}{4\psi}-\frac{\bar{\rho}_{\nu}^2}{4\gamma^2}\right)\sigma^2+\left(\gamma^2\widetilde{\nu}^2-\sigma^2\right)\\
    &\leq\left(\gamma^2\widetilde{\nu}^2-\sigma^2\right)
\end{align*}
\end{proof}

Next, the desired pitch angle vector $\boldsymbol{\theta}$ designed in the high-level layer is directly translated to individual desired pitch angle in the low-level control such that $\boldsymbol{\theta}_d=\boldsymbol{\theta}$. Without loss of generality, the control development in the low-level is performed using a single actuator since the actuator model provided in Eqn.~\eqref{equ:2nd order pitch Act} is the same for all actuators.

Consider the low-level tracking error $\varepsilon=\theta-\theta_d$, where $\theta_d$ is the desired pitch angle for one actuator. Since the pitch actuator dynamics in Eqn.~\eqref{equ:2nd order pitch Act} is significantly faster than the rotor dynamics given in Eqn.~\eqref{eqn:rotor_dyn}, then it is reasonable to assume that $\dot{\theta}_d=0$ for developing the controller for Eqn.~\eqref{equ:2nd order pitch Act}. See \cite{anubi2014new,anubi2013roll} for more detailed analysis where the authors use singular perturbation and time-scale separation to show that it is indeed the case for mechatronic systems of this form.
Next, consider the filtered tracking error
\begin{align}\label{equ:eta}
    z\triangleq\dot{\varepsilon}+2\zeta_0\omega_{n0}\varepsilon.
\end{align}
Taking first derivative of Eqn.~\eqref{equ:eta} yields
\begin{align*}
        \dot{z}=&\Ddot{\varepsilon}+2\zeta_0\omega_{n0}\dot{\varepsilon}\\\nonumber
    =&\Ddot{\theta}+2\zeta_0\omega_{n0}\dot{\theta},
\end{align*}
then inserting the pitch actuator dynamics in Eqn.~\eqref{equ:2nd order pitch Act}, and the fault model in Eqn.~\eqref{equ:fault-model} yields 
\begin{align}\nonumber
            \dot{z} &=\delta\omega_{n0}^2(\theta_r-\theta)-\delta\omega_{n0}^2\left(\frac{2\left(\rho\zeta_0\omega_{n0}-\zeta_0\omega_{n0}\right)}{\delta\omega_{n0}^2}\right)\dot{\theta}\\\label{equ: eta_dotwu}
       &=\delta\omega_{n0}^2\left(\theta_r-\theta-\eta\dot{\theta}\right),
\end{align}
where
\begin{align*}
    \eta\triangleq\left(\frac{2\left(\rho\zeta_0\omega_{n0}-\zeta_0\omega_{n0}\right)}{\delta\omega_{n0}^2}\right)
\end{align*}
is an uncertain parameter due to actuator faults.
Consequently, the low-level control law is designed as
\begin{align}\label{equ:LLCL}
    \theta_r=\theta-k_\theta z+\widehat{\eta}\dot{\theta},
\end{align}
where $k_\theta>0$ is the control gain, and $\widehat{\eta}$ is the parameter estimation. Next, inserting the control law in Eqn.~\eqref{equ:LLCL} into the open-loop dynamics in Eqn.~\eqref{equ: eta_dotwu} yields 
\begin{align}\label{equ:CLL}
\dot{z}=-\delta\omega_{n0}^2k_\theta z+\delta\omega_{n0}^2\widetilde{\eta}\dot{\theta},
\end{align}
where $\widetilde{\eta}=\widehat{\eta}-\eta$ is the parameter estimation error. Consequently, the adaptive law is designed as
\begin{align}\label{equ:adapt law}
    \dot{\widehat{\eta}}=-\alpha z\dot{\theta},
\end{align}
where $\alpha>0$ is an adaptation gain.
\begin{theorem}\label{Thm: theorem2}
Consider the low-level control law in Eqn.~\eqref{equ:LLCL}, together with the adaptive law in Eqn.~\eqref{equ:adapt law}, if the following conditions hold
\begin{align}\label{equ:SC2}
    k_\theta>0,\hspace{2mm}\alpha>0,
\end{align}
then the closed-loop error system in Eqn.~\eqref{equ:CLL} is globally asymptotically stable. 
\end{theorem}
\begin{proof}
Denote the positive definite radially unbounded candidate Lyapunov function
\begin{align*}
    V=\frac{1}{2}z^2+\frac{\delta\omega_{n0}^2}{2\alpha}\widetilde{\eta}^2,
\end{align*}
taking first derivative yields
\begin{align*}
    \dot{V}=z\dot{z}+\frac{\delta\omega_{n0}^2}{\alpha}\widetilde{\eta}\dot{\widehat{\eta}},
\end{align*}
Inserting the closed-loop dynamics in Eqn.~\eqref{equ:CLL} results in
\begin{align}\label{equ:bt}
    \dot{V}=-\delta\omega_{n0}^2 k_\theta z^2+\delta\omega_{n0}^2z\widetilde{\eta}\dot{\theta}+\frac{\delta\omega_{n0}^2}{\alpha}\widetilde{\eta}\dot{\widehat{\eta}}.
\end{align}
Next, inserting the adaptation law in Eqn.~\eqref{equ:adapt law} yields
\begin{align}\label{equ:bt}
    \dot{V}&=-\delta\omega_{n0}^2k_\theta z^2.
\end{align}
 This indicates that $\dot{V}$ is negative semidefinite implying that $V\in \mathcal{L}_{\infty}$; which implies that $z\in \mathcal{L}_{\infty}$, and $\widetilde{\eta}\in \mathcal{L}_{\infty}$. Thus, it implies that $\varepsilon$, and $\dot{\varepsilon}$ ($\dot{\theta}$) are also bounded. Now, consider Eqn.~\eqref{equ:CLL}, as $z$, $\widetilde{\eta}$, and $\dot{\theta}$, are bounded, then $\dot{z}$ is also bounded implying that $z$ is uniformly continuous. Next, taking the integral of both sides in Eqn.~\eqref{equ:bt} yields
\begin{align*}
    \int_{0}^{\infty} \dot{V} d\tau&=\int_{0}^{\infty}-\delta\omega_{n0}^2k_{\theta}z^2 d\tau,
\end{align*}
\begin{align}\label{equ:L2_proof}
    V(\infty)-V(0)&=-\int_{0}^{\infty}\delta\omega_{n0}^2k_{\theta}
    z^2 d\tau,
\end{align}
Since $V$ is bounded, then $z\in \mathcal{L}_2$. Since $z\in \mathcal{L}_2\cap \mathcal{L}_{\infty}$, and also it is uniformly continuous, according to the Barbalat's lemma\cite{khalil2002nonlinear}, $z$ converges to zero asymptotically. Thus, $\varepsilon\rightarrow 0$, and $\dot{\theta}\rightarrow 0$ as $t\rightarrow\infty$.
\end{proof}

%% file: Simulation.tex
The developed controller is validated on a 5MW variable pitch WT model via \textcolor{blue}{the} Fatigue Aerodynamics Structures and Turbulence (FAST) simulator developed by the US National Renewable Energy {\color{blue} Lab} (NREL) \cite{jonkman2009definition}.
\begin{figure*}[t]
    \centering
    \includegraphics[width=17cm,height=8.5cm]{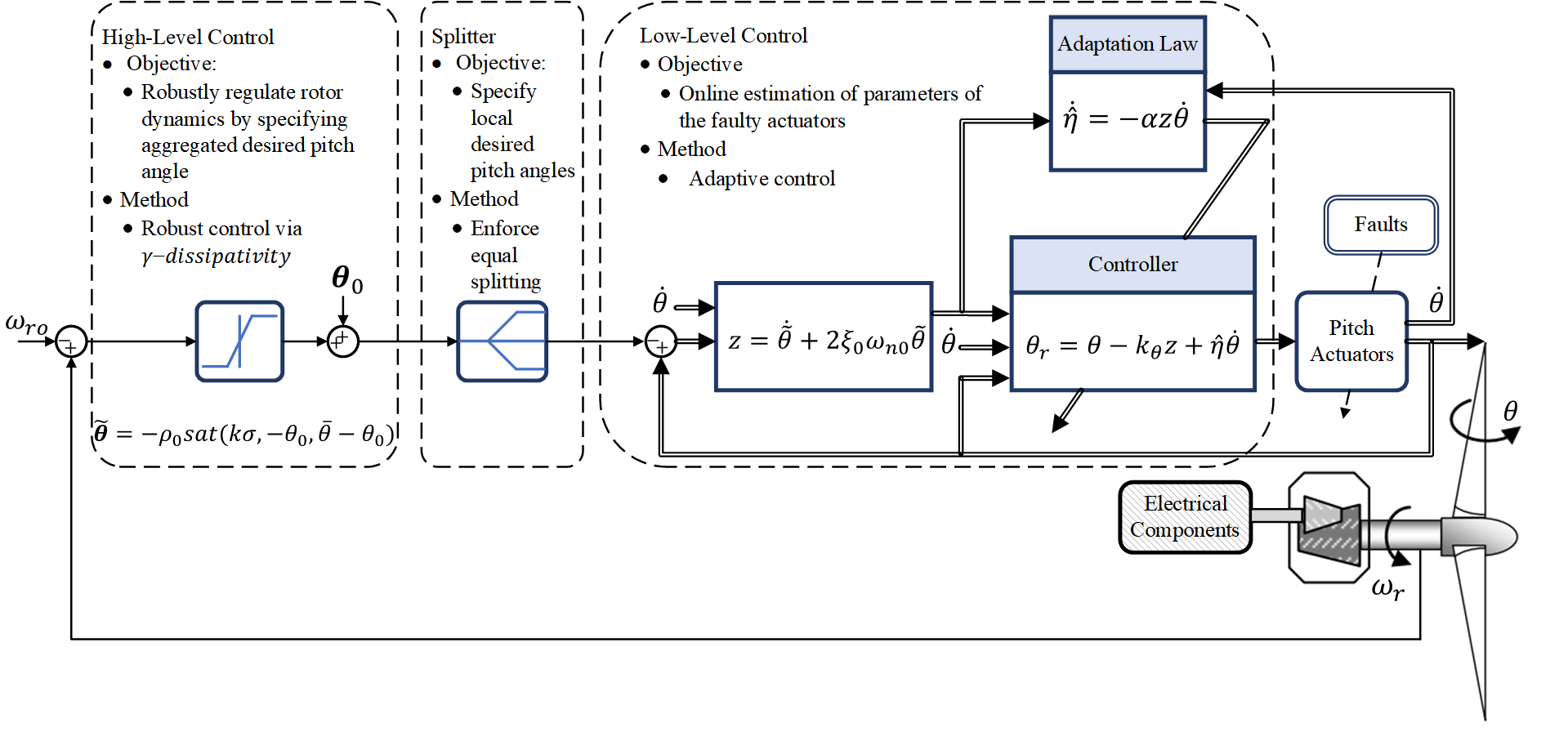}
    \caption{Block diagram of the developed controller}
    \label{fig: Block diagram}
\end{figure*}
\begin{figure}[t]
    \centering
    \includegraphics[width=8.5cm,height=3cm]{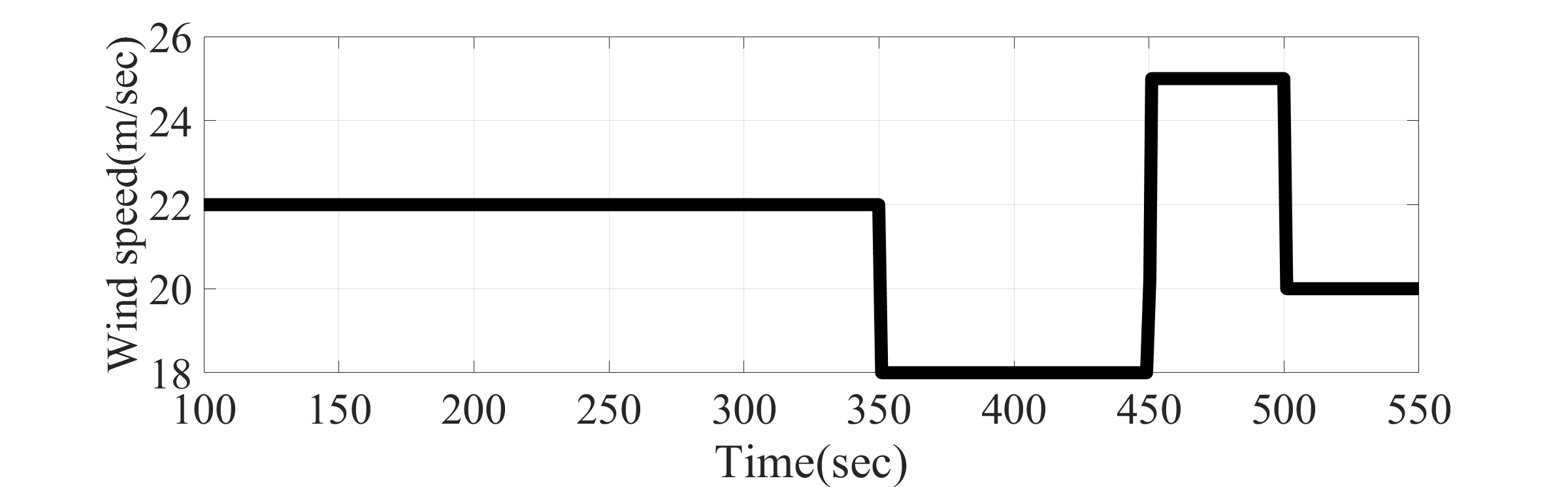}
    \caption{Step wind signal}
    \label{fig:Wind_stp}
\end{figure}
  \begin{figure}[t]
    \centering
    \includegraphics[width=8.5cm,height=4cm]{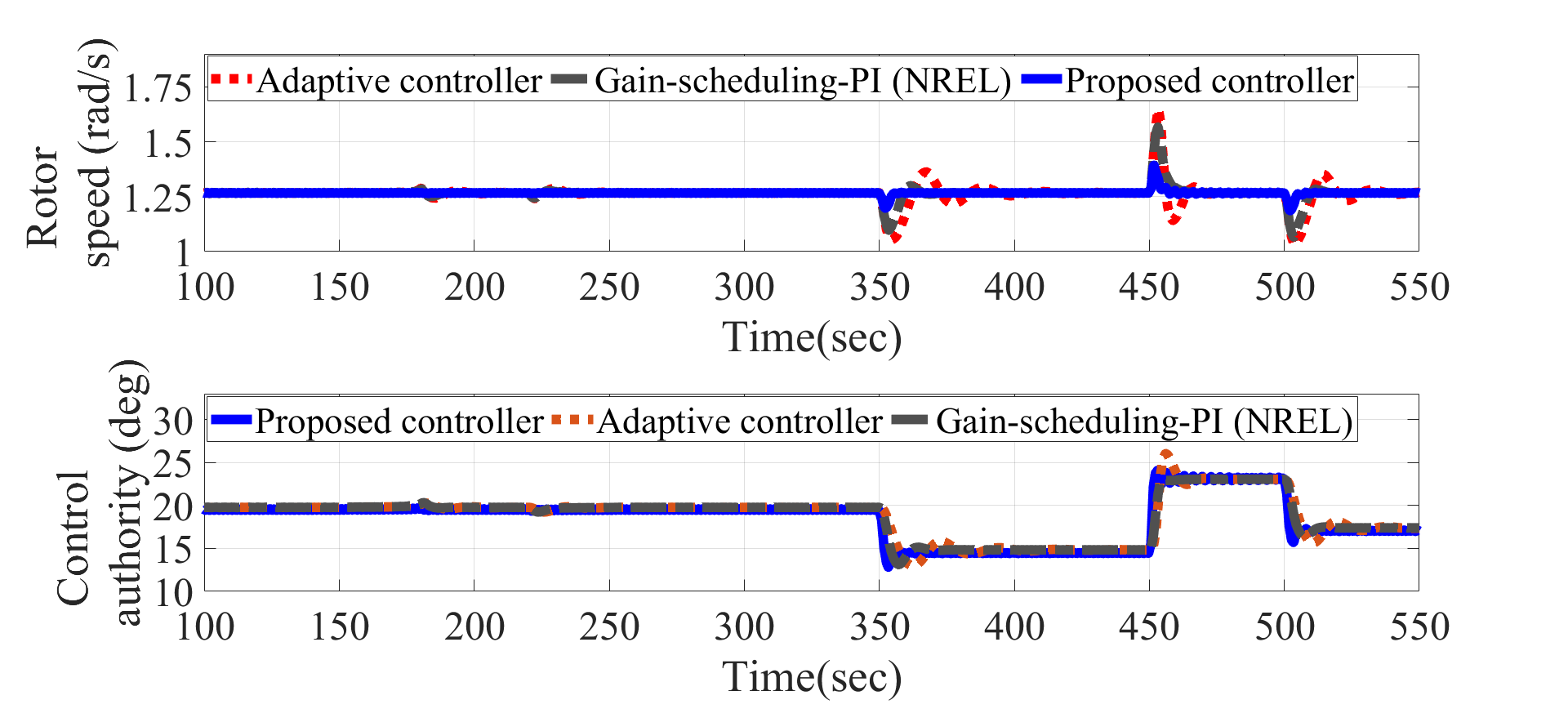}
    \caption{Rotor speed response, and control authority subjected to the step wind signal}
    \label{fig:rs_stp}
\end{figure}
\begin{figure}[t]
    \centering
    \includegraphics[width=8.5cm,height=3cm]{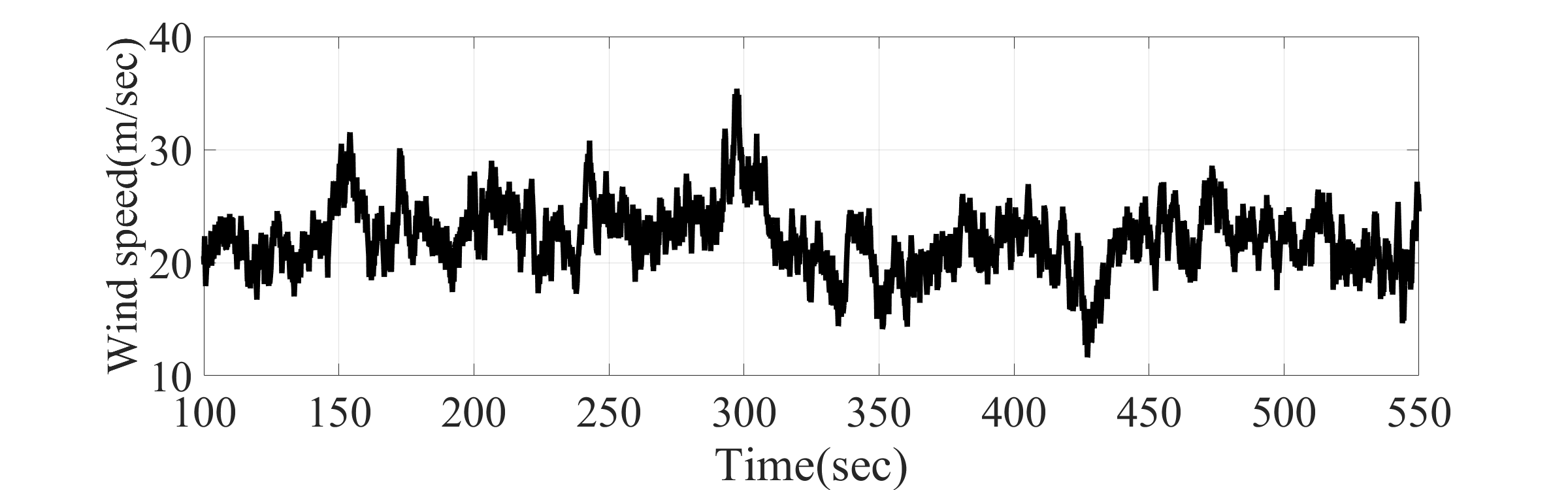}
    \caption{Stochastic wind signal}
    \label{fig:Wind_sth}
\end{figure}
\begin{figure}[t]
    \centering
    \includegraphics[width=8.5cm,height=4.5cm]{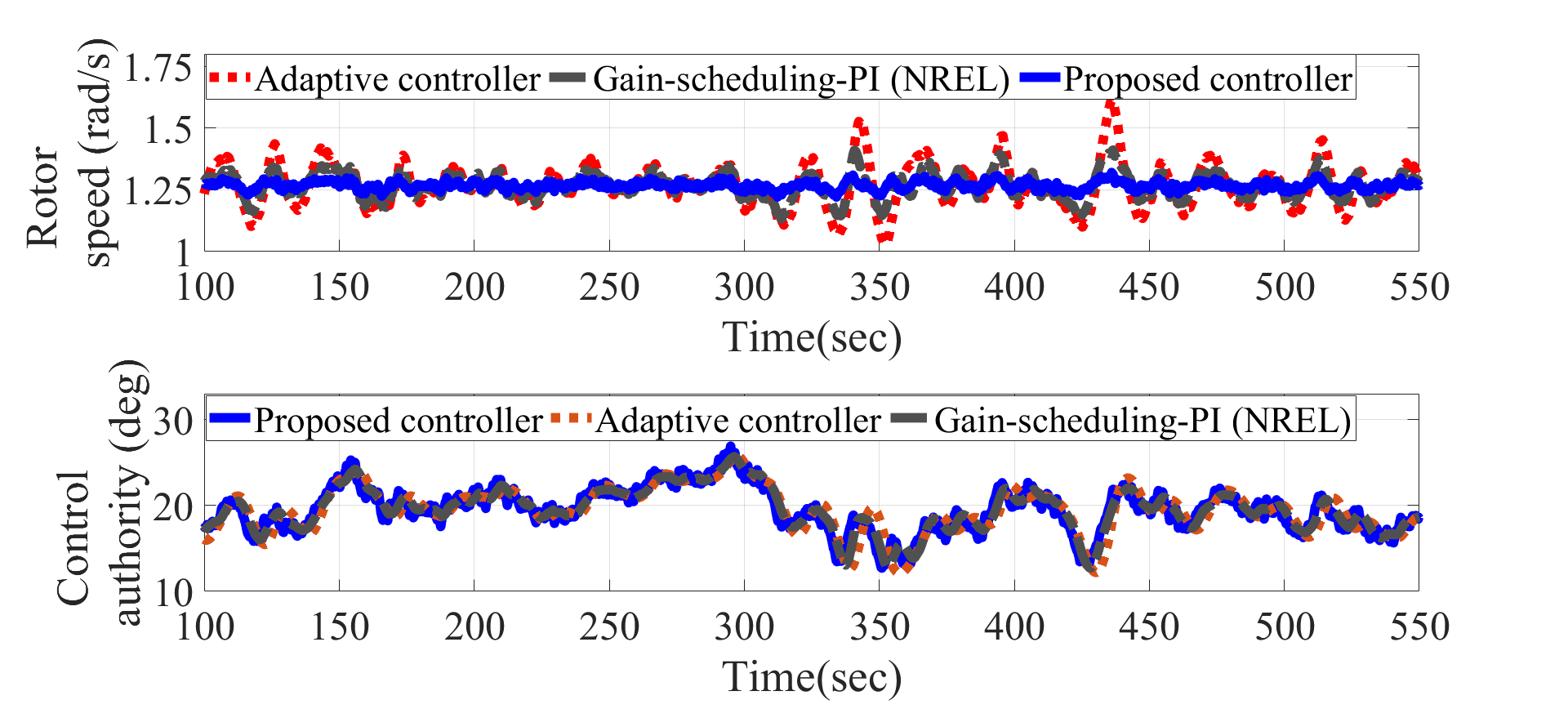}
    \caption{Rotor speed response, and control authority subjected to the stochastic wind signal}
    \label{fig:rs_sth}
\end{figure}
The desired value for the rotor speed is $\omega_{r0}=1.267rad/s$, and the operating points for the pitch angle, and the wind speed are $\theta_0=19.94 deg$, $\nu_0=22m/s$.
 The nominal actuator parameters are $\zeta_0=0.6$, $\omega_{n0}=11.11rad/s$. The parameters $\kappa=7622.7$, $J=43784700 kg.m^2$. {\color{blue}Regarding the parameters $p_1$, $p_2$, and $p_3$ in Eqn.~\eqref{eqn:g1}, and Eqn.~\eqref{eqn:g2}, their values depend on the operating point of the WT. In ~\cite{wasynczuk1981dynamic}, their values are obtained experimentally. Here, their best values at the operating point are obtained via solving an optimization problem, which minimizes the deviation from the experimental values subject to \eqref{eqn:rotor_dyn}:
 \begin{align*}
     &\textsf{Minimize: } \left\|\mathbf{p}-\bar{\mathbf{p}}\right\|^2\\
     &\textsf{Subject to:}\\
     &\left(\frac{\kappa \nu_0^3}{J\omega_{r0}}\left(\frac{\nu_0}{\omega_{r0}}-p_1\right)-\frac{\kappa \nu_0^3}{3J\omega_{r0}}p_3\left\|\boldsymbol{\theta}_0\right\|^2\right)\check{p}_2-\frac{P_0}{J\omega_r}=0,
 \end{align*}
 where $\mathbf{p}=\begin{bmatrix}p_1&\check{p}_2&p_3\end{bmatrix}$, with $\check{p}_2=\textsf{e}^{-p_2\frac{\nu_0}{\omega_{r0}}}$, are the optimization variables, and the vector $\bar{\mathbf{p}}$ contains the corresponding reported experimental values. The solution} gives $p_1=5.4148$, $p_2=0.0682$, and $p_3=0.029$. Since the wind speed is bounded in the full-load region as $11.4 m/s\leq v\leq 25m/s$, and the pitch angle is also bounded as $\theta\in[0^\circ,90^\circ]$, the upper bound parameters are calculated as $\bar{\rho}_{\nu}=1$, $\bar{\rho}_{\omega}=1.5$. Moreover, considering that $\boldsymbol{\rho}_\theta=\frac{-2 \kappa\eta_{\nu}^3 p_3}{3J\eta_{\omega}}\textsf{e}^{\left(-p_2\frac{\eta_\nu}{\eta_\omega}\right)}\boldsymbol{\eta}_\theta$, and choosing $\boldsymbol{\rho}_0=[-1\hspace{2mm}-1\hspace{2mm}-1\hspace{2mm}]^\top$, then $\boldsymbol{\rho}_\theta^\top\boldsymbol{\rho}_0\geq0.15$ is obtained. Consequently, for a given $\gamma=0.25$, and a fixed $\psi=0.5$, and then applying the sufficient condition in Eqn.~\eqref{equ:SC} gives the condition for the control gain as $k\geq 54.17$ in the high-level control loop, which the choice of $k=55$ is chosen. {\color{blue} Regarding ~\eqref{equ:SC}, decreasing $\gamma$ increases the gain $k$, and since $\left\|\widetilde{\omega}_r\right\|\leq\left\|\sigma\right\|\leq\gamma\left\|\widetilde{v}\right\|$, decreasing $\gamma$ attenuates the impact of the wind disturbance on the rotor speed. Hence, increasing $k$ decreases the rotor speed fluctuations against the wind disturbance.} The design parameters for the low-level controller are chosen such that the sufficient condition in Eqn.~\eqref{equ:SC2} is satisfied, then the design parameters are chosen as $\alpha=0.3$, $k_\theta=2.5$. {\color{blue}Note that $k_\theta$ adjusts the speed of the error convergence while $\alpha$ tunes the speed of the parameter estimation.}
 The control structure of the developed controller is illustrated in Fig.~\ref{fig: Block diagram}. 
The proposed controller is compared with the well-tuned PI-gain-scheduling controller (baseline controller) developed by NREL \cite{jonkman2009definition}, and an adaptive controller proposed in \cite{simani2012adaptive}, where in the later an adaptive directional forgetting scheme is utilized to identify the model. Two cases, with deterministic and a stochastic wind signals, are used to show the performance of the proposed control system. The stochastic wind signal has a mean value of $22m/s$, and the turbulence intensity of $20\%$ produced by the TurbSim software,~\cite{jonkman2009turbsim}.
 The fault is ramped up in the time interval $150s-180s$, and then completely activated within $180s-220s$, and ramped down within $220s-250s$. Figure~\ref{fig:Wind_stp} shows the deterministic wind signal applied to the WT. Figure~\ref{fig:rs_stp} shows the rotor speed response subjected to the deterministic wind signal. 
\begin{figure}[t]
    \centering
    \includegraphics[width=8.5cm,height=4cm]{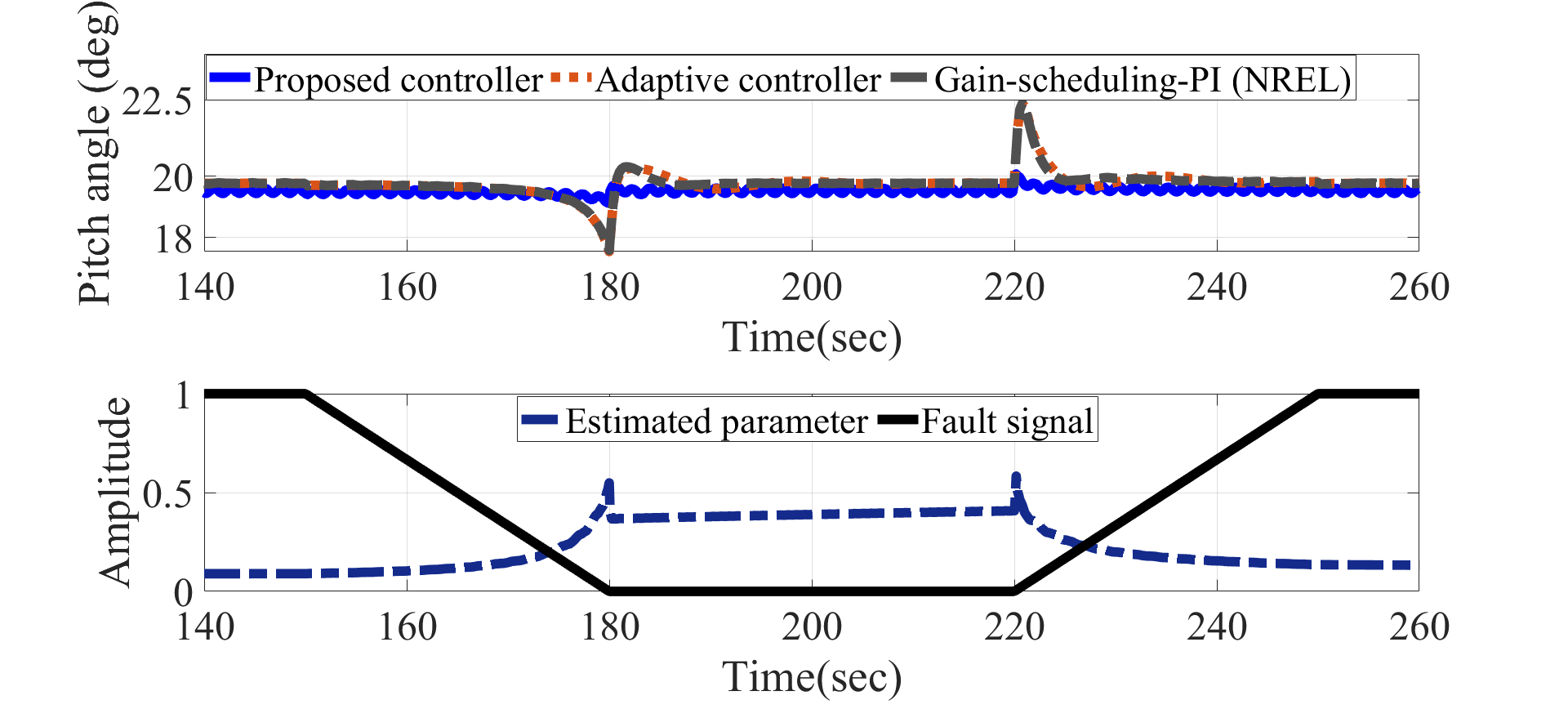}
    \caption{Faulty pitch actuator and the estimated parameter subjected to the step wind signal}
    \label{fig:pitch_stp}
\end{figure}
\begin{figure}[t]
    \centering
    \includegraphics[width=8.5cm,height=4cm]{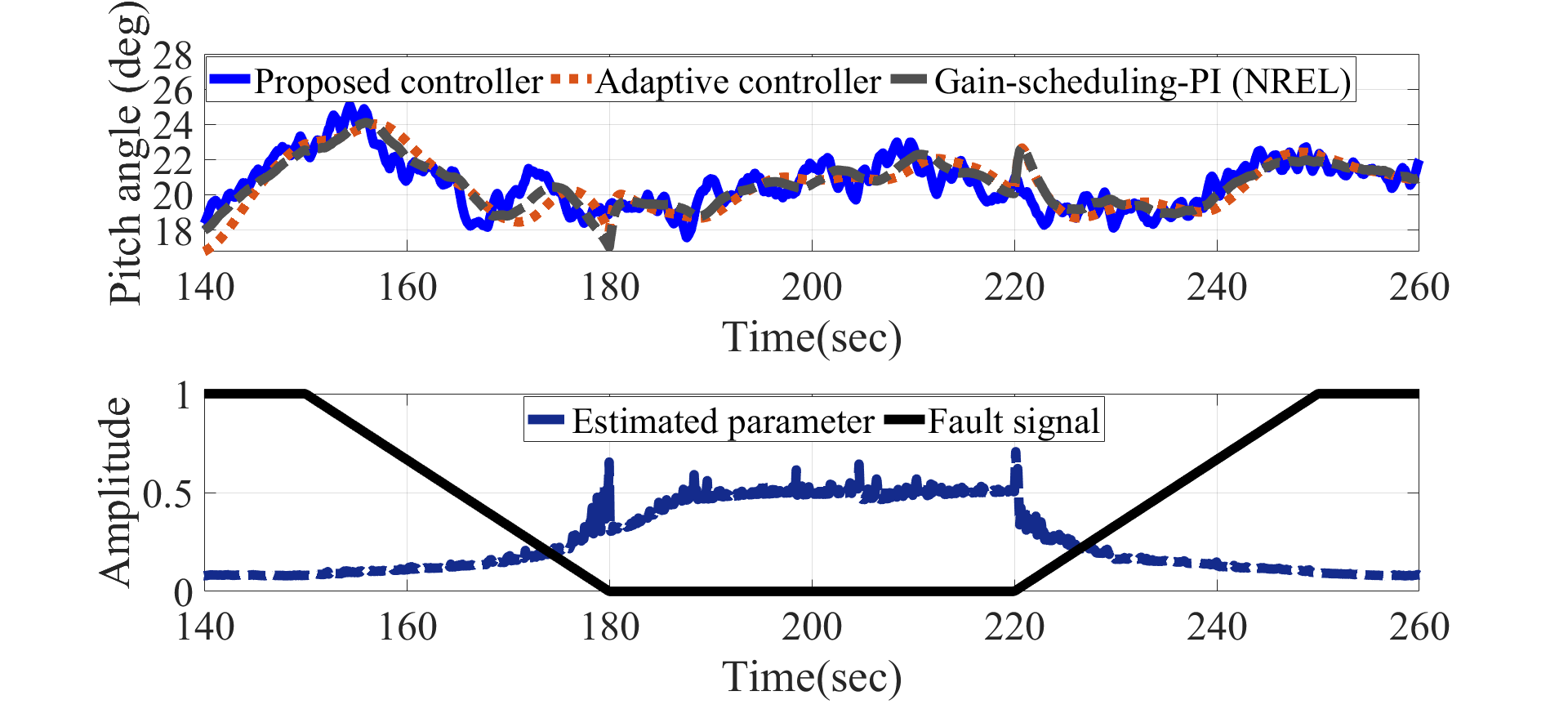}
    \caption{Faulty pitch actuator and the estimated parameter subjected to the stochastic wind signal}
    \label{fig:pitch_sth}
\end{figure}
It clearly illustrates that the proposed controller outperforms others in rejecting the actuator fault, and has significantly less fluctuations in the rotor speed, when the wind speed is changing. Figure~\ref{fig:Wind_sth} shows the stochastic wind signal applied to the WT.
Figure.~\ref{fig:rs_sth} shows the rotor speed response, and the control authority to the stochastic wind signal. It shows that other methods have huge fluctuations between 300s-450s due to a sudden drop in the wind speed. Note that the rotor speed responses are noisy.

Figure.~\ref{fig:pitch_stp} (deterministic), and Fig.~\ref{fig:pitch_sth} (stochastic) illustrate the results of the low-level closed-loop system.
The results show that when the fault happens, the adaptive parameter estimation starts to increase, and when the fault is linearly vanishing, at the same time, the adaptive parameter converges to its faultless value. Hence, the adaptive mechanism rejects this time-varying-incipient fault. Also, Fig.~\ref{fig:pitch_sth} shows that the proposed controller has more pitch activities to deal with the stochastic wind signal. 
Table \ref{tab:rms} shows the root mean square (rms) of the rotor speed error with respect to the PI-gain-scheduling controller. It shows that the developed controller significantly reduces the rotor speed tracking error. Note that the numbers are relative to that of the NREL's baseline controller and does not represent the absolute error for each controller. This was done because the absolute numbers for the proposed controller are really small. Table~\ref{tab:rmsfault} compares the rms of the error of the rotor speed with respect to the fault-free baseline controller during the time that the fault can happen, i.e, between $150sec-250sec$. Table~\ref{tab:rmsfault} shows that the rms of the error for the proposed controller is not only the smallest, but also has no increment when fault occurs compared to other controllers. Figure~\ref{fig:rms chart} compares the rms of the rotor speed error subjected to stochastic wind signals with different mean values, the TI of $20\%$, and $24\times 24$ girds. It shows that the proposed controller attenuates the wind disturbance and has very small increment in the rms of the error when the wind speed increases.  
\begin{table}[h!]
   \caption{Relative Rms of the rotor speed error compared with NREL's baseline controller}
    \label{tab:rms}
    \begin{center}
    \begin{tabular}{c l l}
    \hline
    Wind type&Deterministic&Stochastic\\
 \hline\hline
 Baseline Controller~\cite{jonkman2009definition}&100\%&100\%\\
 Adaptive Controller~\cite{simani2012adaptive}&142.68\%&175.59\%\\
 Proposed Controller &29.85\%&28.97\%\\
 \hline
    \end{tabular}
    \end{center}
\end{table}
\begin{table}[h!]
   \caption{Relative Rms of the rotor speed error compared with the fault-free NREL's baseline controller subjected to stochastic wind signal}
    \label{tab:rmsfault}
    \begin{center}
    \begin{tabular}{c l l}
    \hline
    Fault condition&Fault-free&Faulty\\
 \hline\hline
 Baseline Controller~\cite{jonkman2009definition}&100\%&102.65\%\\
 Adaptive Controller~\cite{simani2012adaptive}&136.34\%&141.91\%\\
 Proposed Controller &34.22\%&34.22\%\\
 \hline
    \end{tabular}
    \end{center}
\end{table}
\begin{figure}[t]
    \centering
    \includegraphics[width=8cm,height=4.5cm]{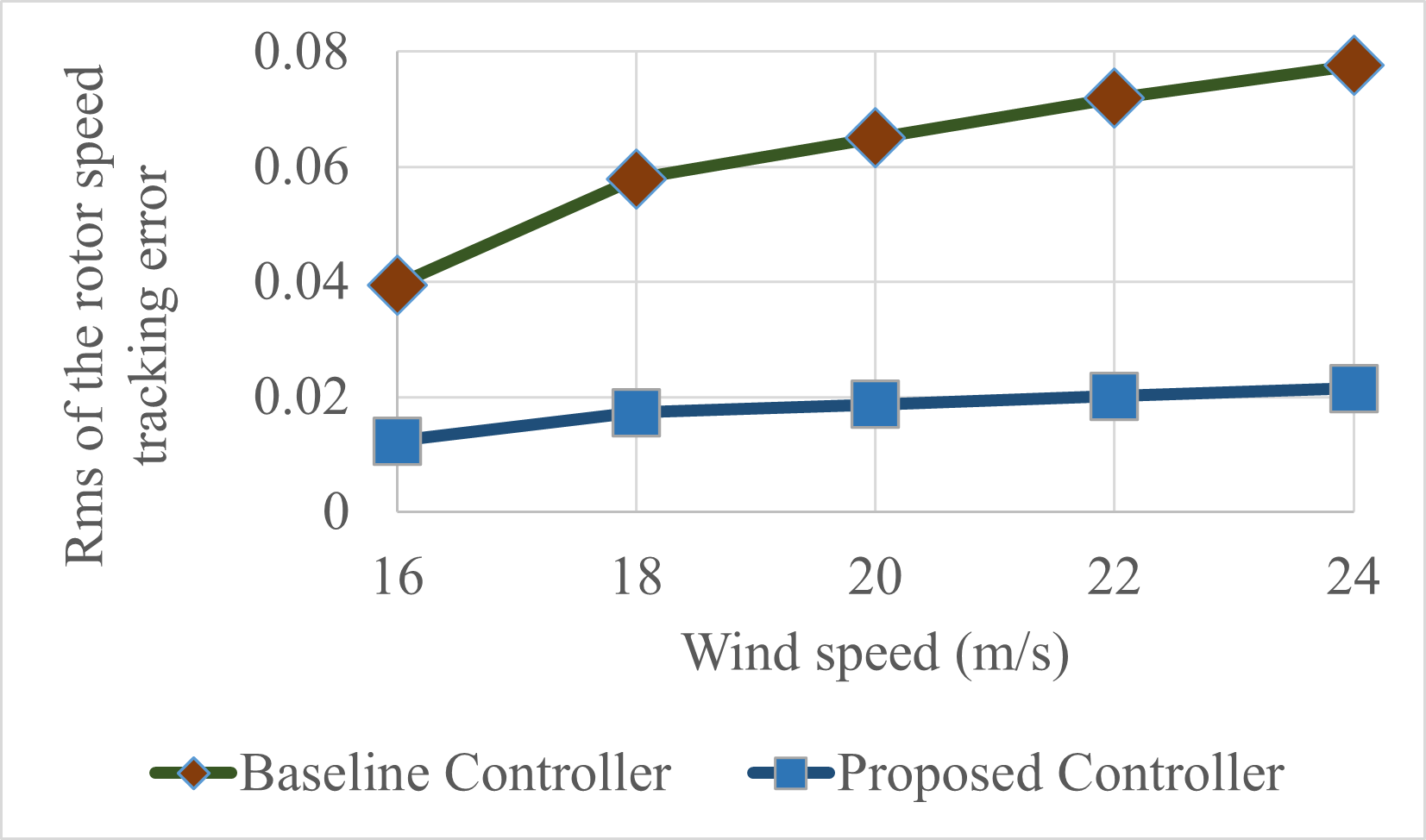}
    \caption{Root means square of the rotor speed tracking error subjected to stochastic wind signals at different operating points}
    \label{fig:rms chart}
\end{figure}

%% file: Conclusion.tex
This paper has proposed a nonlinear robust-adaptive controller for a WT with faulty pitch actuators. It was shown that the proposed $\mathcal{L}_2$ controller provides a finite-gain $\mathcal{L}_2$ stable mapping from the wind disturbance to the rotor speed tracking error and the proposed adaptation mechanism ensures a globally asymptotically stable pitch angle error in the low-level layer with time-varying uncertainty in the control input effectiveness. 
Simulation results show a considerable reduction in fluctuations of the output in spite of exposure to an unmeasurable largely fluctuating wind disturbance and time-varying actuator faults. 

%% file: Future_work.tex
While this paper designed an effective robust-adaptive controller, it did not address the problem for individual actuators. Thus, to improve the efficiency of the controller, faulty actuators should be indentified, and then a splitter should be designed to distribute the control authority from the high-level control to the low-level layer based on the degree of the faults. In addition, more aggressive faults, and actuator failure, will be considered.